\documentclass[12pt,twoside,reqno]{amsart}
\usepackage{amsmath, amsthm, amscd, amsfonts, amssymb, graphicx, color}
\usepackage[utf8]{inputenc}
\usepackage{cleveref}
\crefname{section}{§}{§§}
\Crefname{section}{§}{§§}
\let \oldsection
\section
\renewcommand{\section}{\vspace{8pt plus 4pt}\oldsection}

\newcommand{\beqa}{\begin{eqnarray*}}
\newcommand{\eeqa}{\end{eqnarray*}}
\newcommand{\beqn}{\begin{eqnarray}}
\newcommand{\eeqn}{\end{eqnarray}}

\newcommand{\R}{\mathbb R}

\newcommand{\M}{\mathbb M}

\newcommand{\mcH}{\mathcal H}
\newcommand{\mcB}{\mathcal B}

\newcommand{\m}{\mu}

\usepackage{lipsum}
\usepackage[super]{nth}
\usepackage{xcolor}
\usepackage{tikz}
\usepackage{pgfplots}
\usepackage{mathrsfs}
\usepackage{graphics}
\usepackage{graphicx} 
\usepackage{epsfig}
\usepackage{wrapfig}
\definecolor{olive}{rgb}{0.3, 0.4, .1}
\definecolor{fore}{RGB}{249,242,215}
\definecolor{back}{RGB}{51,51,51}
\definecolor{title}{RGB}{255,0,90}
\definecolor{dgreen}{rgb}{0.,0.6,0.}
\definecolor{gold}{rgb}{1.,0.84,0.}
\definecolor{JungleGreen}{cmyk}{0.99,0,0.52,0}
\definecolor{BlueGreen}{cmyk}{0.85,0,0.33,0}
\definecolor{RawSienna}{cmyk}{0,0.72,1,0.45}
\definecolor{Magenta}{cmyk}{0,1,0,0}
\newtheorem{thm}{Theorem}[section]

\newtheorem{corollary}[thm]{Corollary}
\newtheorem{lemma}[thm]{Lemma}
\newtheorem{prop}[thm]{Proposition}

\theoremstyle{definition}
\newtheorem{defn}{Definition}[section]

\theoremstyle{remark}

\newtheorem{example}{Example}[section]

\numberwithin{equation}{section}
\begin{document}
\begin{center}\Large{{\bf{ Henstock-Orlicz space and its dense space}}} \\
%(Dedicated to $45^{th}$ Birthday of my Advisor Prof. Bipan Hazarika) 
%\vspace{0.5cm}
 
    Bipan Hazarika$^{1, \ast}$ and Hemanta Kalita$^{2}$ \\
%\vspace{0.5cm}
%$^{1}$Department of Mathematics, Patkai Christian College (Autonomous),
%Dimapur, Patkai 797103, Nagaland, India\\
%$^{2}$Department of Electrical Engineering, Computer Engineering and Mathematics, Howard University, Washington DC 20059, USA\\
%$^1$Department of Mathematics and Computer Sciences, via Vanvitelli, 1  I-06123 Perugia, Italy
$^{1}$Department of Mathematics, Gauhati University, Guwahati 781014, Assam, India\\
E-mail: bh\_rgu@yahoo.co.in; bh\_gu@gauhati.ac.in\\
%$^{2}$  Department of Mathematics, Gauhati University, Guwahati 781014, Assam, India
$^{2}$Department of Mathematics, Patkai Christian College (Autonomous), Dimapur, Patkai 797103, Nagaland, India\\
%$^{3}$Department of Mathematics, Rajiv Gandhi University, Rono Hills, Doimukh 791112, Arunachal Pradesh, India\\
Email:    hemanta30kalita@gmail.com%bh\_rgu@yahoo.co.in;  
%\title[Geometric difference sequence spaces over Non-Newtonian Calculus]{Geometric difference sequence spaces over Non-Newtonian Calculus}
%\author[K. Boruah]{Khirod Boruah}
%\author[ B. Hazarika]{Bipan Hazarika\\
 %       Department of Mathematics,\\
  %      Rajiv Gandhi University\\
   %     Doimukh-791112, India}
%\thanks{$^{\ast}$ The corresponding author.}
\end{center}
\title{}
\author{}
\thanks{$^{\ast}$The corresponding author}
\thanks{\today} 
\begin{abstract}
The motivation of the article is to introduce Henstock-Orlicz space with non-absolute integrable functions. We prove $ C_{0}^{\infty} $ is dense in the Henstock-Orlicz space, which is not dense in the classical Orlicz space.
{\footnotesize {\bf{AMS subject classification \textrm{(2010)}:}}} 26A39, 46B03, 46B20, 46B25, 46T12.
\end{abstract}
\maketitle

\maketitle

\pagestyle{myheadings}
\markboth{\rightline {\scriptsize   Hazarika, Kalita}}
        {\leftline{\scriptsize Henstock-Orlicz space and its dense space }}

\maketitle

    \section{Introduction and Preliminaries}
    The  Lebesgue integration theory was developed in the year 1904 by Hendri Lebesgue. This theory and the function $x^p$ in the definition of $L^p$ space inspired by Z.W. Birnbaum and W. Orlicz to proposed a generalized space of $L^p,$  later on it was known as  Orlicz space. This space was later  developed by Orlicz himself.       The fundamental properties of Orlicz space with Lebesgue measure found in \cite{Ma}. In \cite{Mm}, Rao and Ren discussed  the theory of Orlicz space which is a more generalized version of $L^p$-space with the help of Young functions and the underlying measure. The basic ideas of the proofs of the theorems of Orlicz space were analogues of the basic results on $ L^p$-space.        In   \cite{Tv}, Thung presented a translation invariant subspace $ L_1(\R^n) \cap L_{\phi}(\R^n) $ to be dense in Orlicz space $ L_{\phi}(\R^n) $ by over coming difficulties as $ C_{0}^{\infty}(\R^n)$ is dense in $ L^p(\R^n) $ but not generally dense in $ L_{\phi}(\R^n). $ On the other hand,  Donaldson and Trudinger \cite{Th} proved that $ C_{0}^{\infty}(\R^n) $ is dense in Orlicz-Sobolev space.        Henstock integral was first developed by R. Henstock and J. Kurzweil independently, during 1957-1958 from Riemann integral with the concept of tagged partitions and gauge functions. Henstock-Kurzweil integral (we write Henstock integral) is a kind of non-absolute integral and contain Lebesgue integral (see \cite{Rh,Rh1,Kurzweil,Cs}).\\
       Based on the articles \cite{Th,Ra,Rh,Cs,Bs},  we introduce Henstock-Orlicz  (shortly H-Orlicz) space with the concept of Henstock integrable functions in place of bounded measurable functions with compact support. Throughout, the paper our functions are of Lebesgue measurable. It is known that,  if $f$ is  bounded  with compact support, then following are equivalent:
       \begin{enumerate}
       	\item[(a)] $f$ is Henstock–Kurzweil integrable, 
       	\item[(b)]	$f$ is Lebesgue integrable,
       	\item[(c)] 	$f$ is Lebesgue measurable.
       \end{enumerate}
       In general, every Henstock–Kurzweil integrable function is measurable, and $f$ is Lebesgue integrable if and only if both $f$ and $|f|$ are Henstock–Kurzweil integrable. This means that the Henstock–Kurzweil integral can be thought of as a "non-absolutely convergent version of Lebesgue integral". So, our space is equivalent to the classical Orlicz space. The Orlicz space is motivated us to construct the H-Orlicz space. Also, we note some important difference such as  $ C_{0}^{\infty}(\R^n) $ is dense in H-Orlicz space. \\    
        Recall, a tagged partition of an interval $I=[a,b]$ is a finite set of ordered pairs $D=\{(t_i, I_i)~: 1 \leq i \leq m \},$ where $\{I_i~: 1 \leq i \leq m \}$ is a partition of $I$ consisting of closed non-overlapping subintervals and $t_i$ is a point in $I_i;~t_i$ is called the tag associated with $I_i.$ If $f:I \to \R,$ the Riemann sum of $f$ with respect to $D$ is defined as
        \begin{equation}\label{eq01}
        S(f,D) = \sum_{i=1}^{m} f(t_i)l(I_i),
        \end{equation}% $$ S(f,D) = \sum_{i=1}^{m} f(t_i)l(I_i),$$ 
        where $l(I_i) $ is the length of the subinterval $I_i.$\\
        If $\delta: I\to  (0,\infty )$ is a positive function, we define an open interval valued
        function on $I$ by setting  $\gamma(t) = (t-\delta(t),t + \delta(t)).$ If $I_i = [x_i,x_{i+1}],$ and $t_i\in I_i, 1\le i\le m,$ we can rewrite 
        \begin{equation}\label{eq02}
        t_i\in I_i\subset \gamma(t_i).
        \end{equation} 
        A function $\gamma$ defined on  $I$ such that  $\gamma(f)$ is an open interval containing $t$
        for each $t\in I$ is called a gauge on $I.$ If $D= \{(t_i, I_i) : 1 \le i \le m\} $ is a tagged
        partition of $I$ and $\gamma$ is a gauge on $I,$ we say that $D$ is $\delta$-fine if (\ref{eq02}) is satisfied.         i.e. if $D$ is $\delta$-fine, we write $D<< \delta.$ 
   \begin{defn}\cite[Definition 2]{Cs} 
   A function $ f:[a,b] \to \mathbb{R} $ is said to be  Henstock integrable over $[a,b]$ if there exists $ A \in \mathbb{R} $  such that for every  $ \epsilon > 0 $ there exists a gauge $\delta: [a,b] \to \mathbb{R} $ such that $|S(f,D)-A| < \epsilon $ whenever $D$  is a $\delta$-fine  tagged partition of $[a,b].$\\
   Or A function $ f:[a,b] \to \mathbb{R} $ is Henstock integrable if there exists a function $ F:[a,b] \to \mathbb{R} $  such that  for every $ \epsilon > 0 $ there exists a function $ \delta(t) > 0 $ such that for any $ \delta$-fine partition $ D=\{[u,v] ,t \} $ of $I=[a,b], $ we have
    $$\left\| \sum[f(t)(v-u)]-F(u,v)\right\| < \epsilon, $$ 
    where the sum $ \sum $  runs  over $ D= \{ ([u,v], t) \}$ and $F(u,v)= F(v)-F(u).$     We write $ H\int_{I} f=F(I).$
   \end{defn}
   Throughout the article, we assume an abstract measure space $(\mathcal{B} , \Sigma, \mu_{\infty}),$ where $\mathcal{B}$ is some set of points and $\Sigma $ is an $\sigma$-algebra of its subsets on which a $\sigma$-additive function $ \mu_{\infty}: \Sigma \to \R^+ $ is given. Also we assume that the measure $\mu_{\infty}$ is the Lebesgue measure. We say that this measure space has the finite subset property if for every $E \in \Sigma$ with $\m_\infty(E) = \infty$ there exists a family of subsets $\{E_i\}_{i=1}^{\infty} \subset \Sigma $ with $E_i \subset E;~0 < \m_\infty(E_i) < \infty$ and $\m_\infty\left(\bigcup\limits_{i=1}^{\infty}E_i\right) = \infty.$\\
   This give us 
   \[\mu_{\infty}(E)=\left\{\begin{array}{c} 
   0 ,\mbox{ if } E = \emptyset,\\%n-\left[ \left\vert \lambda _{n}\right\vert %
   %\right] +1\leq k\leq n ;\\ 
   +\infty,\mbox{ if } E\neq \emptyset \end{array}\right .\]
   %        $\mu_{\infty}(E) = 0 $ if $ E = \Phi,~~= +\infty $ if $  E \neq \Phi. $
   Otherwise it does not restrict the generality of our assumption.
%\section*{Definition 1.1} A function $ f:[a,b] \to R $ is Henstock integral if there exist $ A \in R $ such that for $ \epsilon > 0 $ there exist a gauge $\delta:[a,b] \to R $ such that for each tagged partition $(p, (c_k)_{k=1}^{n})$ that is $\delta(x)$ fine, $$|R(f,p)-A| < \epsilon $$ Or A function $ f:[a,b] \to R $ is Henstock integrable if there exists a function $ F:[a,b] \to R $ such that for every $ \epsilon > 0 $ there is a function $ \delta(t) > 0 $ such that for any $ \delta-$ fine partition $ D=\{[u,v] : t \} $ of $[a,b], $ we have $$|| \Sigma [f(t)(v-u)-F(u,v)]|| < \epsilon $$ where the sum $ \Sigma $ is understood to be over $ D= \{ ([u,v], t) \}$ and $F(u,v)= F(v)-F(u) $ We write $ H\int_{I_0} f=F(I_0)$\\
%\section*{ Definition 1.2}
        We denote $ H(\m_\infty) $ be the space of all Henstock integrable functions defined on $ \mcB.$  In  \cite{Ra,Cs,Bs}, authors proved that $ H(\m_\infty)$ is a vector space under the usual operations of pointwise addition and scalar multiplication on $\R.$ In the one-dimensional case, Alexiewicz \cite{AL} has shown that the class  of Henstock  integrable functions, with respect to the  norm defined  in the following manner: 
         for $f \in H(\R)$, define $\left\| f \right\|_H$ by
\[
\left\| f \right\|_H  = \sup_s \left| {\int_{ - \infty }^s {f(r)d(r)} } \right|.
\]
 is a normed space, and it is known that $H(\R)$ is not complete (see \cite{AL}). Gill and Zachary \cite{Tl} introduced  two new classes of Banach space of Henstock integrable functions. 
         \\
    {\bf Henstock-integral on $\R^n$}: Let $\mathbb{R}^n$ be $n$-dimensional Euclidean space. A typical element of $\mathbb{R}^n$ will be denoted by $x=(x_1,x_2,\dots,x_n).$ An interval in $\mathbb{R}^n$ is a set of the form $J=[\mathbf{a},\mathbf{b}] := \prod_{i=1}^n[a_i, b_i],$ where
    $-\infty < a_i < b_i <\infty $ for $i = 1,2\dots,n.$. The set $\prod_i^n[a_i, b_i]\subset \mathbb{ R}^n$ is known as a degenerate interval if $a_i =b_i$ for some $i\in \{1,2,\dots,n\}.$ Two intervals $J=[\mathbf{a},\mathbf{b}], K= [\mathbf{u}, \mathbf{v}] $ in $\mathbb{R}^n$ are said to be non-overlapping if  $\prod_i^n(a_i, b_i)\cap \prod_i^n(u_i, v_i)$ is empty.   Also, union of
    two intervals in $\mathbb{R}^n$ to be an interval in $\mathbb{R}^n.$ (see Lemma 2.1.2 \cite{lee}).        We know that the space  $\R^n $ equipped with the maximum norm $||.||,$ where %however for Henstock integral the maximal norm 
    $||x|| = \max\limits_{ 1 \leq k \leq n}|x_k|.$  With this norm, we denote the closed ball of $\mathbb{R}^n$ by $B[\mathbf{x},r]=\{\mathbf{x}\in\mathbb{R}^n: ||\mathbf{y}-\mathbf{x}||\le r\},$ whose center is  $\mathbf{x}$ with sides parallel to the co-ordinates axes of length $2r.$ It is a closed interval    for side $i$ about $ x_i$ is in $[a_i,b_i].$ So, let $B[\mathbf{x},r]=[J,\mathbf{x}],$ where $J= \prod_{i=1}^{n}[a_i,b_i]$, $J $ is closed interval in $\R^n.$ 
    \begin{defn}
    If $E$ is a compact ball in $\R^n,$ a partition $P$ of $E$ is a collection $\{(J_i,x_i): x_i \in J_i, 1 \leq i \leq m \},$ where $J_1, J_2,\dots,J_m$ are non overlapping closed intervals i.e. $\m_\infty[J_i \cap J_j] = 0, i \neq j$ and $\bigcup\limits_{i=1}^{m} J_i = E.$
    \end{defn}
     If $\delta $ is a positive function on $E$ we say $P$ is Henstock $\delta$-partition of $E$ if for each $i$, $J_i \subset B^{'}(x_i, \delta(x_i)).$ The function $\delta $ is a gauge on $E.$
     \begin{defn}
     A function $f: E \to \R$ is said to be Henstock integrable on $E,$ if there exists a number $A$ such that for any $\epsilon >0 $ there exist a gauge $\delta$ and Henstock $\delta$-partition on $E$ such that $$|\sum\limits_{i=1}^{m}f(x_i)\m_\infty(J_i)- A| < \epsilon.$$ 
     \end{defn}
     %Lebesgue integral is a special case of Henstock integral( See Lemma 3.23 to Theorem 3.24 of \cite{Tl}. We introduce Henstock-integrable function in classical Orlicz spaces; we called this new spaces as H-Orlicz spaces.We assume that  our measure $\mu_{\infty} $ have the finite subset property. We consider the finite subset property as:

       % \begin{defn}
        %Let $(\mathcal{B}, \Sigma, \m_\infty)$ be a measure space of infinite measure. We say that this measure space has the finite subset property if for every $E \in \Sigma$ with $\m_\infty(E) = \infty$ there exists a family of subsets $\{E_i\}_{i=1}^{\infty} \subset \Sigma $ with $E_i \subset E;~0 < \m_\infty(E_i) < \infty$ and $\m_\infty(\cup_{i=1}^{\infty}E_i) = \infty.$
        %\end{defn} 
        %This give us 
         %\[\mu_{\infty}(E)=\left\{\begin{array}{c} 
          % 0 ,\mbox{ if } E = \emptyset,\\%n-\left[ \left\vert \lambda _{n}\right\vert %
          %\right] +1\leq k\leq n ;\\ 
          %+\infty,\mbox{ if } E\neq \emptyset \end{array}\right .\]
        %        $\mu_{\infty}(E) = 0 $ if $ E = \Phi,~~= +\infty $ if $  E \neq \Phi. $
         %Otherwise it does not restrict the generality of our assumption. %Text Book of M.M. Rao already define atom and non atomic for $\mu$. We are taking the same concept for our work.\ 
              
 \begin{defn}
  \cite{Mm} A function $ \theta: \mathbb{R}^{+} \to \mathbb{R^+}$ is said to be Young function, so that $\theta(x)= \theta(-x), \theta(0)=0, \theta(x) \to \infty $ as $ x \to \infty,$ but $ \theta(x_0) = +\infty $ for some $ x_0 \in \mathbb{R} $ is permitted. 
  \end{defn}

  \begin{defn} \cite{Ma}  Let $ m : \mathbb{R}^+ \to \mathbb{R}^+ $ be non-decreasing right continuous and non-negative function satisfying $$ m(0)=0, \mbox{~and~}  \lim\limits_{t \to \infty } m(t) = \infty.$$
  	\end{defn}
  \begin{defn} \cite{Ma} 
  	 A function $\theta: \mathbb{R}^{+} \to \mathbb{R} $ is called an $N$-function if there is a function $'m' $  satisfying the above sense that $$ \theta(u) = \int_{0}^{|u|} m(t) dt.$$
   Evidently, $\theta$ is an $N$-function if it is continuous, convex, even satisfies $$ \lim\limits_{u \to \infty } \frac{\theta(u)}{u}= \infty  \mbox{~and~}  \lim\limits_{u \to 0} \frac{\theta(u)}{u}= 0 .$$
    For example $\theta_p(x)= x^p;~ p > 1.$
   
     \end{defn} 
   % \section*{ Definition 1.8 }
    \begin{defn}
        
          An $N$-function $\theta $ is said to satisfy $\Delta_2$-condition if there is a $ k > 0 $ such that $ \theta(2x) \leq k \theta(x)  $ for large values of $ x.$   %The  $\Delta_2$-condition  is called  regular (globally) if there is a $k>0$ so that $ \theta(2x) \leq k \theta(x), x \geq x_0 \geq 0.$%~(x_0=0) $ for some absolute constant $k>0.$
          %\end{enumerate} 
                  \end{defn} 
              
              \begin{defn}
              	An Orlicz function is a function $\theta: \mathbb{R}^{+} \to [0,\infty]$ such         that
              	\begin{enumerate}
              		\item[(1)]  		$\theta(0)=0$, $\theta(x)>0$ for some $x>0$ and $\theta(x_1)<\infty$ for some $x_1>0$
              		\item[(2)] $\theta$ is increasing: $x_1\le  x_2\iff \theta(x_1)\le \theta(x_2)$
              		\item[(3)] $\theta$ is convex: $\theta(ax_1 + (1-a)x_2)\le a\theta(x_1) +(1-a)\theta(x_2), 0\le a\le 1$
              	\item[(4)]  $\theta$ is left-continuous: $\theta(x-)=\theta(x), x\ge  0.$
              	\end{enumerate}              \end{defn}
              \noindent Throughout our work,  we consider only real functions in  H-Orlicz class  $\mathcal{H}^{-\theta}(\mu_{\infty}).$  We define $\mathcal{H}^{-\theta}(\mu_{\infty}) $ in Definition \ref{de21} of section 2.\\%$2.1.$ 
      %\noindent         
    We refer \cite{etaltin,Ma,Mm,BP,nsrnns,bcrg} for details on Orlicz space and its applications. 
        \begin{defn}
        A Banach space $X$ is called Banach lattice if $||f||_{X} \leq ||g||_{X}$ for every $f,g \in X$ such that $ |f| \leq |g|.$
        \end{defn}
      \noindent   For every non-negative measurable function $f:\mathbb{R}^{+}\to \mathbb{R}^{+},$ we define the function \begin{equation}\label{eq001}
    \eta_{f}(t):=\m_\infty(\{f>t\}), t\in \mathbb{R}^{+},
    \end{equation} where  $\{f>t\}:=\{  x\in\mathbb{R}^{+}:f(x)>t \}, t\in \mathbb{R}^{+}$ are the upper Lebesgue sets of $f .$\\
    The function $\eta_{f}: \mathbb{R}^{+}\to [0,\infty]$, defined by (\ref{eq001}), is called     the (upper) distribution function of $f.$ Clearly $\eta_{f}$  is smooth (infinitely differentiable) and $\eta_{f}$  has compact support (is identically zero outside some bounded interval).\\
   The non-negative functions $f$ and $g$ are called equimeasurable if $\eta_f=\eta_g$ i.e., $\m_\infty(\{f>t\})=\m_{\infty}(\{g>t \}), t\in \mathbb{R}^{+}.$ %    mff > yg D mfg > yg for all y 2 OE0;1/:

     %   The upper distribution function $\eta_f $ of a non-negative measurable function $f $ is defined by upper Lebesgue set$\{ f > t\}$ :$$\eta_f(t)= \m_\infty\{f> t \}= \m_\infty \{x :~f(x) > t \},~t \in \R^{+}.$$
     Let $f: \mathbb{R}^{+}\to \mathbb{R}$ be a real measurable function on  $\mathbb{R}^{+}$  and let $\eta_{|f|}(t):=\m_{\infty}(\{ |f|>t \})\in [0,+\infty], t\in\mathbb{R}^{+} $  
     be the distribution function of the absolute value $|f|.$ It is possible to have $\eta_{|f|}(t)= +\infty.$ In      the sequel, we assume (unless otherwise stated) that $\eta_{|f|}(t) \ne +\infty.$
     Since the function $\eta_{|f|}$ is decreasing and right-continuous, it has a unique      generalized inverse, which is also decreasing and right-continuous. This inverse  function $\eta_{|f|}^{-1}$ will be denoted by 
     \[ \eta_{|f|}^{-1}(x)=\inf\{ t\ge 0: \eta_{|f|}(t)\le x \}  .\]
     Since $\eta_{|f|}(t) \ne +\infty,$ there exists $t_0\ge 0$ such that the value $\eta_{|f|}(t_0)=\m_{\infty}(\{ |f|>t_0 \})< +\infty.$  This means that $\lim\limits_{t\to \infty}\eta_{|f|}(t)=0$ and hence  $\eta_{|f|}^{-1}(x)<\infty$ for all $x>0.$ \\
     A nonzero Banach space $(X,||.||_{X})$ of measurable functions on $(\mathbb{R}^{+},\m_{\infty})$  is called {\it symmetric} if the following conditions hold:
     \begin{enumerate}
     	\item[(a)]  If $|f|\le |g|$ and $g\in  X,$ then $f\in X$ and $||f||_X\le ||g||_X.$
     	\item[(b)] If $f$ and $g$ are equimeasurable and $g\in X,$ then $f\in X$ and $||f||_X=||g||_X.$
     \end{enumerate}
    A normed space $ (X, ||.||_X)$ satisfying condition (a) is called a {\it normed ideal lattice}. 
   \begin{thm}
    If $ f: I_0 \subseteq  \mathcal{B} \to \mathbb{R} $ is a measurable function with  gauge function  $ \delta: {I_0} \subseteq \mathcal{B} \to \mathbb{R} $ then $ \theta(f):\mathcal{B} \to \mathbb{R^+} $ is  Henstock integrable. 
    \end{thm}
  \begin{proof}
  Let $ f:I_0 \subseteq \mathcal{B}  \to \mathbb{R}^{+} $ be   measurable (integrable) function, then for all $\epsilon > 0 $ there exists a $\delta: I_{0} \to \mathbb{R} $ such that  $$\left|\sum\limits_{(I,w) \in \pi}\sum\limits_{(I^{'}, w^{'}) \in \pi^{'}} [f(w)-f(w^{'})]\m_\infty(I \cap I^{'})\right| < \epsilon $$ for all partitions $\pi$ and $\pi^{'} $ of $I_0$ finer than $\delta.$\\
  As $\theta $ is Young function. So, $\theta(x) \to \infty $ as $ x\to \infty.$ Our claim is $\theta(f) $ is  Henstock integrable.\\
   Since, Young function by definition, is an extended real Borel function. So, $\theta(f) $ is measurable. If $\pi$ and $\pi^{'} $ are both partitions of the same interval of  $I_0,$ then for any subinterval $ I$ of $ I_0$  we can write 
   \begin{align*} 
   &\m_\infty(I) = \sum\limits_{(I^{'},w^{'}) \in \pi^{'}}\m_\infty(I \cap I^{'}).\end{align*} 
   So, \begin{align*}
   &~~~~\sum\limits_{(I,w) \in \pi} \theta(f)(w) \m_\infty(I) =\sum\limits_{(I,w) \in \pi}\sum\limits_{(I^{'},w^{'}) \in \pi^{'}} \theta(f)(w) \m_\infty(I \cap I^{'})\\
&\mbox{i.e.}~\left|\sum\limits_{(I,w) \in \pi} \theta(f)(w) \mu_{\infty}(I)  - \sum\limits_{(I^{'},w^{'}) \in \pi^{'}} \theta(f)(w) \m_\infty(I )\right| < \epsilon.\end{align*}
   Thus, $ \theta(f)$ is Henstock integrable.
  \end{proof}
   \section{Structure of H-Orlicz class}
  In this section we define the Henstock-Orlicz class for bounded measurable functions with compact support and discuss some fundamental results of this class. 
  \begin{defn}\label{de21}
   Let $ \mathcal{H^{-\theta}(\mu_\infty)} $ be the set of all $ f:\mathcal{B} \to \mathbb{R} $   bounded  measurable with compact support  for $\Sigma \subset \mathcal{B}$  such that $\int_\mathcal{B} \theta(|f|)d\mu_{\infty} $ is  Henstock integrable. \\ %\section*{}We now discuss about the structure of $BH^{-\theta}$
  i.e. $\mathcal{H^{-\theta}(\mu_\infty)} = \{ f  \mbox{~is~bounded~measurable~with~compact~support}: \int_\mathcal{B} \theta(|f|)d\mu_{\infty} \in H(\m_\infty)\}.$
  \end{defn}
  \begin{thm}
  The space $ \mathcal{H^{-\theta}(\mu_{\infty}) }$ is absolutely convex. That is if $ f, g \in \mathcal{H^{-\theta}(\mu_{\infty})}$ and $ \alpha, \beta $ are scalars such that  $|\alpha|+|\beta| \leq 1$  then $\alpha f + \beta g \in \mathcal{H^{-\theta}(\mu_{\infty})}.$ Also $ h \in  \mathcal{H^{-\theta}(\mu_{\infty})}, |f| \leq |h|, f$ is  Henstock integrable, then $ f \in \mathcal{H^{-\theta}(\mu_{\infty})}.$
  \end{thm} 
  \begin{proof}
     Let $ f, g \in \mathcal{H^{-\theta}(\mu_{\infty})}.$ 
Then by the monotonicity and convexity of $\theta,$ we get $ 0< \gamma = |\alpha| +|\beta| \leq 1.$        Now, \begin{align*}
       \theta(|\alpha f + \beta g|) &\leq \theta(|\alpha||f| +|\beta||g|) \\
       &\leq \gamma \theta\left(\frac{|\alpha|}{\gamma}|f| + \frac{|\beta}{\gamma}|g|\right) \\ 
       &\leq {|\alpha| \theta(|f|) + |\beta|\theta(|g|)}.
       \end{align*}
The right hand side is  Henstock integrable. So, $\alpha f + \beta g \in H^{-\theta}(\mu_{\infty}).$\\
For second part $ |f| \leq |h|.$   If $f \neq h $ then for all $x \in \mcB,$ there exists $\epsilon >0 $ such that   $ |h-\epsilon| \leq |f| \leq |h|.$\\
%This implies $ |h-\epsilon| \leq |f| \leq |h|.$\\
As $ h \in H^{-\theta}(\mu_{\infty})$ gives $\theta(|h|) $ is  Henstock  integrable. Also, $\theta(|h-\epsilon|) $ is  Henstock  integrable.\\
    So, by Lemma 9 \cite{Cs}, $ f \in \mathcal{H^{-\theta}(\mu_{\infty})}.$
\end{proof}
  
  \begin{thm}
  The space $\mathcal{H^{-\theta}(\mu_\infty)} $ is a linear space if and only if $\theta$ satisfies $\Delta_2$-condition.
   %\cite{Musial} A subset $ A $ of $ L^1(\mu) $ is uniformly integrable if and only if there is an $N$-function $ \phi $ with $ \Delta^{'} $ condition such that $ A $ is relatively weakly compact in $ L^\phi(\mu).$
  \end{thm}
  \begin{proof}
   For linearity, it is sufficient to verify that for each $ f \in \mathcal{H^{-\theta}(\mu_\infty)}, 2f \in \mcH^{-\theta}(\m_\infty) $ since then $nf \in \mathcal{H^{-\theta}(\mu_\infty)}$ for any integer $n.$ Since $ h \in \mcH^{-\theta}(\mu_{\infty}), |f| \leq |h|, f $ measurable, implies $f \in \mcH^{-\theta}(\m_\infty)$ hence for each $\alpha> 0,~\alpha f \in \mathcal{H^{-\theta}(\mu_\infty)}.$ \\
   Then yields,  $\alpha f + \beta g = \gamma\left(\frac{\alpha}{\gamma}f + \frac{\beta}{\gamma}g\right) \in \mathcal{H^{-\theta}(\mu_\infty)},~\gamma= |\alpha| +|\beta| > 0 $ for any $ f, g \in \mathcal{H^{-\theta}(\mu_\infty)}.$\\
   If $ \theta $ is $\Delta_2$-regular, then $\mu_{\infty}(\mathcal{B}) = + \infty$ and $  \theta(2|f|) \leq k \theta(|f|),~k> 0$ implies $ 2f \in \mathcal{H^{-\theta}(\mu_\infty)}.$\\ % with $ f .$\\
   If $\mu_{\infty}(\mathcal{B})< \infty $ then $\theta(2x) \leq k \theta(x)$ for $ x \geq x_0 \geq 0.$\\
   Let
   \[f_1=\left\{\begin{array}{c} 
              f ,\mbox{ if } |f|\leq x_0,\\%n-\left[ \left\vert \lambda _{n}\right\vert %
             %\right] +1\leq k\leq n ;\\ 
             0,\mbox{ otherwise } \end{array}\right .\]
   %    $ f_1 = f $ if $|f| \leq x_0 $ and $ f_1 = 0 $ otherwise, 
    We set $ f_2 = f - f_1 $ that is $ f = f_1 + f_2 $ and
         \begin{align*}\theta(2|f|) = \theta(2|f_1 + f_2|)
& \leq \theta(2|f_1|)+\theta(2|f_2|) ~(\mbox{by~convexity~of~}\theta)\\
& \leq \theta(2|f_1|) +k\theta(2|f_2|) ~(\mbox{by~}\Delta_2-\mbox{condition}).
         \end{align*}
 Therefore $ \int_{\mathcal{B}}\theta(2|f|)d\mu_\infty $ is Henstock integrable as right side is Henstock integrable.\\
    So, $ 2f \in \mathcal{H^{-\theta}(\mu_\infty)}. $  That is $ nf \in \mathcal{H^{-\theta}(\mu_\infty)} $ for any integer $ n $ \\
    Therefore $ \mathcal{H^{-\theta}(\mu_\infty)} $ is linear when $ \theta \in \Delta_2.$\\
    For converse, let $ E \in \Sigma $ be a set of positive measure on which $\mu_\infty  $ is diffuse and that $\theta$ not belong in $ \Delta_2. $\\ 
    We construct $ f \in \mathcal{H^{-\theta}(\mu_\infty)} $ such that $ 2f \notin \mathcal{H^{-\theta}(\mu_\infty)}.$ \\
    If $ 0 < \alpha < \mu_\infty(E) < \infty .$ Then by hypothesis on $\mu_\infty ,$ there is  $ F \subset E,~F \in \Sigma $ with $\mu(F) = \alpha < \infty.$\\
    We construct a function supported by $ F $ to satisfy our assumption. Let $ \theta(\mathbb{R}) \subset \mathbb{R^+ }.$ Since $\theta$ does not satisfies $\Delta_2$-condition.\\
    There exists a sequence $ x_j \geq j $ such that $\theta(2x_j) > n\theta(x_j),~j \geq 1. $\\
    Let $ n_0 $ be an integer such that $$ \sum_{ n \geq n_0}^{} \frac{1}{n^2} < \infty $$ and $ \theta(x_n) \geq 1$ for all $ n \geq n_0.$\\
    This is possible by diffuseness of $\mu_\infty$ on $F$ there is a measurable $ F_0 \subset F $ such that $$\mu_{\infty}(F_0) = \sum_{ n \geq n_0}^{} \frac{1}{n^2} < \infty. $$
    Similarly we can find a set $ D_1 \in \Sigma, D_1 \subset F_0$ such that $\mu_{\infty}(D_1) = \frac{1}{n^2}.$\\
    Since  $\mu_{\infty}(F_0 - D) > 0 ,$
we can again find $ D_2 \in \Sigma $ and $ D_2 \subset F_0 - D_1 $ such that $\mu_{\infty}(D_2) = \frac{1}{(n_0 + 1)^2}.$\\
    Then for disjoint sets $ D_n \in \Sigma,~\mu_{\infty}(D_n) = \frac{1}{(n_0 + n-1)^2},~n \geq 1.$\\
    Let $ F_k \subset D_k,~F_k \in \Sigma,$ be chosen such that $ \mu_{\infty}(F_k) = \frac{\mu_{\infty}(D_k)}{\theta(x_n)}.$\\
    Let $$ \sum_{j=1}^{n} \theta(x_j)\mu_{\infty}(F_j) = \int_{I_0}f. $$ Then $ f \in \mathcal{H^{-\theta}(\mu_{\infty})} $ (Remark 5 of \cite{Cs}).
However         \begin{align*}
                  \int_{\mathcal{B}} \theta(2f) d\mu_{\infty}& = \sum_{j=1}^{n} \theta(2x_j)\mu_{\infty}(F_j)\\
                  &\geq \sum_{n \geq n_0} j  \theta(x_j)\mu_{\infty}(F_j)\\ 
                  & = \sum_{ n\geq n_0} \frac{1}{j} = + \infty .
               \end{align*}
    So, $ 2f \notin \mathcal{H^{-\theta}(\mu_{\infty})}.$\\
    So, our conclusion $ \mathcal{H^{-\theta}(\mu_{\infty})} $ is a class of scalar function and it is linear if and only it is closed under multiplication by positive number.\end{proof}
    \begin{example}
     Consider $\mathcal{B} =\{ [i-1,i]~: i \in \mathbb{N} \} .$ Let $\Sigma $ be power set of $\mathcal{B} $ and $\mu_{\infty}$ be positive measure  with the $i$th interval has measure $1,$ and    let $\theta(x)= e^{x^2} -1.$ Then $\theta(x)$ is $N$-function such that $\theta \notin \Delta_2,$  we have $\mathcal{H^{-\theta}(\mu_{\infty})} $ is linear space.
     \end{example}
     \begin{proof}
      By definition $\theta(x) $ is $N$-function and it is not satisfies $\Delta_2$-condition. 
 We will see $ \mathcal{H^{-\theta}(\mu_{\infty}) }$ is linear space.\\
         If $ f \in \mathcal{H}^{-\theta}(\mu_{\infty}),$ then $\int_\mathcal{B} {\theta(f)} d\mu_{\infty} =\sum\limits_{n=1}^{\infty}(e^{|f|^2}-1) < \infty.$\\
     So, $\int_{\mathcal{B}} \theta(f) d\mu_{\infty}$ is bounded and  Henstock integrable.\\
         Let $ K>0 $ be bound. So $(e^{|f|^2}-1)\leq K $ implies $e^{|f|^2}\leq 1+ K.$\\
        Therefore $\int_{\mathcal{B}} (\theta(2f) d\mu_{\infty} = \sum\limits_{n=1}^{\infty}\left(e^{4|f|^2}-1\right).$\\
        That is 
        \begin{align*}
        \int_\mathcal{B} \theta(2f) d\mu_{\infty} &\leq \sum_{n=1}^{\infty}( \exp(|f(x)|^2)-1)(K+2)((K+1)^2+1)\\&\leq [(K+2)(K+1)^2 +1]\int_{\mathcal{B}}\theta(f)d\mu_{\infty}.
\end{align*}         
           Right side is  Henstock  integrable. So $ 2f \in \mathcal{H^{-\theta}(\mu_{\infty}) }$ and the space is linear.
  \end{proof} 
       \begin{prop}\label{prop21}
       Let $ (\mathcal{B}, \Sigma, \mu_{\infty}) $ be finite measure space. Then $ H(\m_\infty) = \bigcup \{ \mathcal{H^{-\theta}(\mu_{\infty})} : \theta $ ranges over all $N$-function.\}
       %Assume $ f:[a,b] \to X $ is Denjoy-Gel'fand on $[a,t]$ for all $ t \in [a,b) $ and for each $ x \in X. $ The limit $\lim\limits_{t \to b}\int\limits_{a}^{b} xf $ exist, then $f$ is Denjoy-Gel'fand on $[a,b]$ and $$ < x , DG\int\limits_{a}^{b} f > = \lim\limits_{t \to b } < x , DG\int\limits_{a}^{b} f > $$ for each $ x \in X.$
       \end{prop} 
       \begin{proof}
       Since $\theta$ is convex. $\theta(x) \geq ax +b $ for some constant $a,b.$        For each $ f \in \mathcal{H^{-\theta}(\mu_{\infty})} $, we have $\int_{\mathcal{B}}(a|f| + b)d\mu_{\infty} \leq \int_{\mathcal{B}}\theta(|f|)d\mu_{\infty}$.\\
       Since  $ \int_{\mathcal{B}} \theta(|f|)d\mu_{\infty} $ is Henstock integrable, So $f$ is bounded measurable function.  Using Lusin's theorem, and the  continuity of $f$  
       (Exercise 13 of \cite{Ra}), we have  $ f $ is Henstock  integrable, so $ f \in H(\m_\infty).$ Since this hold for all such convex function. Therefore $ \bigcup \mathcal{H^{-\theta}(\mu_{\infty})} \subset H(\m_\infty).$\\
       If $ f \in H(\m_\infty) ,$ and since $f$ is bounded, so $f$ is Lebesgue integrable on $\mathcal{B}$ (see Theorem 9.13, \cite{Ra}). Therefore, we get $|f|$ is Henstock integrable. Thus  $\theta(|f|) $ is Henstock  integrable.\\
       Therefore $ f \in \mathcal{H^{-\theta}(\mu_{\infty})}.$ So $ H(\m_\infty) \subset \bigcup \mathcal{H^{-\theta}(\mu_{\infty})}$ and hence the result.
       \end{proof}
      %\begin{proof}
     % As Theorem 15.12 of \cite{R.A}, $xf$ is Denjoy integrable on $[a,b]$ for all $ x \in X $($ x^* \in X^* $). Let $ c \in [a,b) $ and any sequence $(t_n) $ in $[a,b) $ converges to $b.$\\
            % Define \begin{align*}
             % L_c(x)  &= \lim\limits_{n} \int\limits_{c}^{t_n} xf \\
             %&= \lim\limits_{n} < x, DG\int\limits_{c}^{t_n}f >
     % \end{align*}  
     % The Uniform boundedness Principle gives us $ L_c $ is continuous on $X.$       Then $f$ is Denjoy-Gel'fand on $[a,b].$\\
     % Taking $ c=a , $ we get 
     % \begin{align*}
     % L_a(x) &= \lim\limits_{n} \int\limits_{c}^{t_n} xf \\
             % &= \lim\limits_{n} < x, DG\int\limits_{c}^{t_n}f >.
     % \end{align*} 
     % \end{proof} %Proof: 
 %\section*{Proposition 2.2 }
  \begin{corollary}\label{prop22}
    $ L_1(\mu_{\infty}) = \bigcup \mathcal{H}^{-\theta}(\mu_{\infty}).$ 
  \end{corollary}
% Let $E$ be a perfect set in $[a,b]$ and let $\{I_n\}$ be the sequence of all open interval in $(a,b)$ that intersect $E$ and have rational end points.\\
%For each $n,$ let $ E_n = E \cap I_n .$ \\
% For each positive integers $ m $ and $n ,$ let $ A_{m}^{n} = \{ x \in X : %\int_{E_n}|xf| \leq m \}.$ Then $ X =\cup_n \cup_m A_{m}^{n}.$\\
 %We claim that each of the set $ A_{m}^{n} $ is closed.\\
 %Let $x$ be a limit point of $ A_{m}^{n}$ and let$\{x_k\} $ be a sequence in $ A_{m}^{n} $ that converges to $x.$\\
% Then the sequence $\{ |x_k f|\} $ converges pointwise on $[a,b]$ to the function $|xf| $ and by Fatou's lemma, we have $$\int_{E_n} |xf| \leq \lim\limits_{n \to \infty }\{ \int_{E_n}|x_k f |\} \leq m. $$
% This gives  $ x \in A_{m}^{n} $ and so, $ A_{m}^{n} $ is closed.\\
 %By the Baire category theorem there exist $ M , N , x_0 $ and $\rho > 0 $ such that $\{ x : ||x-x_0|| \leq \rho \} \subset A_{M}^{N} .$ \\
% For each $ x \in X $ with $||x|| \neq 0 ,$ we find $$
 %\int_{E_N }|xf| \leq \frac{||x||}{\rho}\{ \int_{E_N}|\frac{\rho}{||x||}xf + x_0 f | + \int_{E_N}|x_0 f| \}\\
% \leq \frac{2M}{\rho}||x||.$$
% Therefore for each $ x \in X ,$ the function $ xf $ is Lebesgue integrable on $ E \cap I_N.$\\
% Therefore $f$ is Gel'fand on $ E \cap I_N $ and consequently $f$ is Gel'fand on $E.$
 %Proof:  
 \section{H-Orlicz Space}
 In this section we introduce the Henstock-Orlicz (shortly H-Orlicz) space over an arbitrary measure space $ (\mathcal{B}, \Sigma, \mu_{\infty}),$ and discuss some topological properties of the space. 
 \begin{defn}
%Let $ \mathcal{H}^{-\theta}(\mu_{\infty}) $ be the set on 
Let $ (\mathcal{B}, \Sigma, \mu_{\infty})$ be an arbitrary measure space and let $ \mathcal{H}^{-\theta}(\mu_{\infty}) $ be the Henstock-Orlicz class on $ (\mathcal{B}, \Sigma, \mu_{\infty}).$ Then the space $ \mathcal{H}^{\theta}(\mu_{\infty}) $ of all measurable  function $ f: \mathcal{B} \to \mathbb{R} $ such that $ \alpha f \in \mathcal{H}^{-\theta}(\mu_{\infty}) $ for some $ \alpha > 0 $ is called H-Orlicz space. We define the  H-Orlicz space as follows: $$ \mathcal{H}^{\theta}(\mu_{\infty}) = \left\{ f: \mathcal{B} \to \mathbb{R} \mbox{~measurable}:  \int_{\mathcal{B}}\theta(\alpha f ) d\mu_{\infty} \in H(\m_\infty)  \mbox{~for~some~}  \alpha > 0\right\}.$$ 
\end{defn}
  \begin{prop}\label{prop23}
  $ \mathcal{H}^{\theta}(\mu_{\infty}) $ is a linear space.
  \end{prop} %\\ Proof:
  \begin{proof}
  Let $ f_i \in \mathcal{H^{\theta}(\mu_{\infty})},~i= 1, 2. $ Then there exist $\alpha_i $ such that by definition $\alpha_i f_i \in \mathcal{H^{-\theta}(\mu_{\infty})}.$\\
  Let $ \alpha = \min (\alpha_1, \alpha_2). $ Then for $ \alpha > 0 $.
    $$ \int_{\mathcal{B}} \theta\left(\frac{\alpha}{2}(f_1 + f_2 )\right)d\mu_{\infty} \leq \frac{1}{2}\left[\int_{\mathcal{B}}\theta(\alpha_1f_1)d\mu_{\infty} + \int_{\mathcal{B}}\theta(\alpha_2 f_2)d\mu_{\infty}\right].$$
    Right hand side is Henstock-integrable. So, $ f_1 + f_2 \in \mathcal{H^{\theta}(\mu_{\infty})} $ as $ \frac{\alpha}{2} > 0. $\\
     If $ f = f_1 =f_2 $ implies $2f \in \mathcal{H^{\theta}(\mu_{\infty})}. $ So, $ nf \in \mathcal{H^{\theta}(\mu_{\infty})} $ for all integer $ n > 1.$\\
     Therefore $ \beta f \in \mathcal{H^{\theta}(\mu_{\infty})} $ for any scalar $\beta$. Hence $\mathcal{H^{\theta}(\mu_{\infty})}$ is a linear space.
  \end{proof} % \\
  %\section*{Proposition 2.4}
  \begin{defn}
    We define the Luxemburg  norm on $\mathcal{H^{\theta}(\mu_{\infty})}$ as follows: $$ {H_{\theta}} (f)= \inf \left\{ a >0 : H\int_{\mathcal{B}} \theta\left( \frac{f}{a}\right)d\mu_{\infty} \leq 1  \right\}.$$ 
  \end{defn}
  It is understood that $\inf(\emptyset)= + \infty.$ It is the gauge of the set $$U= \{ f~{\textit{measurable}}: H\int_{\mcB}\theta(f)d \m_\infty \leq 1 \}.$$
  \begin{prop}\label{prop320}
  	For each $f \in \mathcal{H}^{\theta}(\m_\infty),$ there is an $\alpha>0$ such that$$ B_{\theta}=  \{\alpha f = g \in \mathcal{H}^{-\theta}(\m_\infty) : H \int_{\mcB}\theta(g) d \m_\infty \leq 1 \}$$ is a circled solid subset of $\mathcal{H}^{- \theta}(\m_\infty).$
  \end{prop}
  \begin{proof} %  {\bf{Proof:}}
  Let $f,g \in \mathcal{H}^{\theta}(\m_\infty).$ Then there exist $\alpha, \beta>0 $ such that $\alpha f, \beta g \in \mathcal{H}^{-\theta}(\m_\infty).$  Let $\nu = \min(\alpha, \beta)$. Then for $\nu>0 $ and using the fact of convexity and monotonicity of $\theta,$ we get $$ H \int_{\mcB}\left( \frac{\nu}{2}(f+g)\right) d \m_\infty \leq \frac{1}{2}[ H\int_{\mcB} \theta(\alpha f) d \m_\infty) + H\int_{\mcB} \theta(\beta g) d \m_\infty].$$ The right side is Henstock integrable. As $\frac{\nu}{2}> 0 ,$ gives $f + g \in \mathcal{H}^{\theta}(\m_\infty).$ In particular, with each $ f $ in $\mathcal{H}^{\theta}(\m_\infty),~2f \in \mathcal{H}^\theta(\m_\infty) $  and then $n f \in \mathcal{H}^\theta(\m_\infty)$ for all integers $n >1$, so that  $\gamma f \in \mathcal{H}^\theta(\m_\infty)$ for any scalar $\gamma.$ Therefore the given set is solid and circled. To hold $\gamma f \in \mathcal{H}^\theta(\m_\infty)$ for some $\gamma>0. $ Let $a_n \to 0 $ be arbitrary and set $\gamma_n= \min(\gamma, a_n).$ Then $\gamma_n \to 0 $ and $\theta(\gamma_n f) \leq \theta(\gamma f)$ and $\theta(\gamma_nf) \to 0 $ as $\theta $ is a continuous Young function. Dominated convergence theorem give us $H\int_{\mcB} \theta(\gamma_n f) \to 0 $ so that for some $n_0,$ we have $H\int_{\mcB}\theta(\gamma_{n_o} f ) d \m_\infty \leq 1.$ Thus $\gamma_{n_0} f \in B_{\theta}.$ %In our assumption $\theta(x) = + \infty $ for $ x>x_0>0,$ then all $p's$ must be bounded then the result trivially hold.
  \end{proof} 
  \begin{thm}\label{prop24}
   The space $ ( \mathcal{H^{\theta}}(\mu_{\infty}), ||.||_{H_{\theta}}) $ is normed linear space.
  \end{thm} 
  \begin{proof}
  %\begin{enumerate}
   (a) If $ f =0 $ a.e., then $ H_{\theta}(f)=0.$\\
  Conversely, if $ H_{\theta}(f) = 0.$\\
  Let $ |f| > 0, $ on a set of positive measure. If possible then there exists a number $ \delta > 0 $ such that $ A =\{ p : |f(p)| \geq \delta \}$ satisfies $ \mu_{\infty}(A) > 0. $\\
  Now \begin{align*}
\theta(n\delta )\mu_{\infty}(A) = & \int_{A}\theta(n\delta)d\mu_{\infty} \\ & \leq \int_{A} \theta(nf)d\mu_{\infty} \\& \leq \int_{\mathcal{B}} \theta(nf)d\mu_{\infty},~n \geq 1.
\end{align*}  Since $\mu_{\infty}(A) > 0 $ and $\theta(n\delta )d\mu_{\infty} \to \infty$ as $ n \to \infty. $ This is not possible so, $\mu_{\infty}(A) =0.$ 
 Thus $ f =0 $ a.e. Therefore $ H_{\theta}(f) = 0 $ if and only if $ f =0 $ a.e.\\
 (b)  $ H_{\theta}(\alpha f) = |\alpha|H_{\theta}(f),~\alpha \in \mathbb{R} $ or $ \alpha \in \mathbb{C}.$\\
  Let $ \alpha \neq 0 .$
 \begin{align*}
 H_{\theta}(\alpha f) &= \inf\left\{ k> 0 : H\int_{ \mathcal{B}}\theta\left(\frac{\alpha f}{k}\right)d\mu_{\infty} \leq 1  \right\}\\& = |\alpha| \inf \left\{ \frac{k}{|\alpha|} > 0 : H\int_{\mathcal{B}}\theta\left(\frac{f}{\frac{k}{|\alpha|}}\right)d\mu_{\infty} \leq 1\right\} \\& = |\alpha|H_{\theta}(f).
\end{align*}  
(c)  $ H_{\theta}(f + g) \leq H_{\theta}(f) + H_{\theta}(g).$\\
Let $ f, g \in \mcH^{\theta}(\m_\infty),  a> H_\theta(f) $ and $b >H_\theta(g)$ Then $ 0<a< \infty$ and $0<b< \infty,$ and suppose $c=a+b >0.$ Since $f+g \in \mcH^\theta(\m_\infty), H_\theta(f+g) < \infty.$ Now 
\begin{align*}
H\int_{\mcB} \theta\left(\frac{f+g}{c}\right)d \m_\infty &= H\int_{\mcB} \theta\left[\frac{f}{a}.\frac{a}{c} + \frac{g}{b}.\frac{b}{c}\right]d\m_{\infty} \\& \leq \frac{a}{c} H\int_{\mcB} \theta\left(\frac{f}{a}\right)d \m_\infty + \frac{b}{c} H\int_{\mcB} \theta\left(\frac{g}{b}\right)d \m_\infty \\&\leq \frac{a}{c} + \frac{b}{c}\leq 1.
\end{align*} 
By Proposition \ref{prop320}, we have $\frac{1}{c}(f+g)\in B_{\theta}$ so that   $H_\theta(f+g) \leq c = a+b.$ If $a \to H_\theta(f) $ and $b\to H_\theta(g),$ we get the triangle inequality $ H_{\theta}(f + g) \leq H_{\theta}(f) + H_{\theta}(g).$ 
Hence  $(\mathcal{H^{\theta}}, ||.||_{H_{\theta}}) $ is a normed linear space.
    \end{proof}
    %\section*{Theorem 2.1}
     \begin{lemma}\label{lemma21}
      For $ \mu_{\infty}(\mathcal{B}) < \infty.$ Assume $\mu_{\infty}$ is bounded, then $\mathcal{H}^{\theta}(\m_{\infty}) \subset L_1(\mu_{\infty}), $ where the inclusion from $ \mathcal{H}^{\theta}$ to  $(L_1, ||.||_1) $ is continuous.
     \end{lemma}
   \begin{proof}
   We have for $ u > 0 $ and $ v \geq 0 $ such that $ \theta(s) \geq us - v $ for all $ s \geq 0 $ implies $ us \leq \theta(s) + v.$\\
    Let $ f \in \mathcal{H^{\theta}(\mu_{\infty})}.$ Then for small enough $ a > 0 $ \begin{align*}
a \int_{\mathcal{B}}|f|d\mu_{\infty} & \leq \frac{1}{u}\int_{\mathcal{B}}[\theta(af) + v]d\mu_{\infty} \\& = \frac{1}{u}\int_{\mathcal{B}}\theta(af)d\mu_{\infty} + \frac{v\mu_{\infty}(\mathcal{B})}{u} \\ & < \infty.
\end{align*} 
Therefore $ f \in L_1(\mu_{\infty}),$ so $ \mathcal{H^{\theta}(\mu_{\infty})} \subset L_1(\mu_{\infty})$ and this inclusion is continuous since with $c >0$ taken sufficiently small for the inequality $\frac{u}{c} - v \m_\infty(\mcB) \geq 1 $ to hold, we have
\begin{align*}
H\int_{\mcB} \theta\left(\frac{f}{c||f||_1}\right)d \m_\infty &\geq \frac{u}{c}H\int_{\mathcal{B}}\frac{f}{||f||_1}d \m_\infty - v \m_\infty(\mcB) \\&= \frac{u}{c} - v \m_\infty(\mcB) \geq 1.
\end{align*}
Hence $c||f||_1 \leq ||f||_{H_\theta}.$
        \end{proof}
\begin{corollary} 
$ \mathcal{H}^{\theta}(\m_{\infty}) \subset H(\m_\infty). $
\end{corollary}
\begin{lemma}\label{lemma22}
 Let $ (f_n)_{n \geq 1} $ be a sequence in $ \mathcal{H^{\theta}(\mu_{\infty})} .$ Then the followings are equivalent:
 \begin{enumerate}
 \item[(i)]  $ \lim\limits_{n \to \infty} ||f_n||_{H_\theta} = 0 $
 \item[(ii)]For all $ a > 0, \lim\limits_{n \to \infty} \sup H\int_{\mathcal{B}} \theta(a f_n)d\mu_{\infty} \leq 1 $
 \item[(iii)]For all $ a>0, \lim\limits_{n \to \infty} H\int_{\mathcal{B}} \theta(a f_n)d\mu_{\infty} =0 $
 \end{enumerate}
\end{lemma}
\begin{proof}
For the proof of  (i) implies (ii):\begin{align*}& \lim\limits_{n \to \infty}||f_n||_{H_\theta} = 0 \\
&\implies \lim\limits_{ n \to \infty} \inf \left\{ a>0: H\int_{\mathcal{B}}\theta\left(\frac{f_n}{a}\right)d\mu_{\infty} \leq 1\right\} = 0 \\
& \implies \inf\left\{ a>0 : \lim\limits_{n \to \infty}H\int_{\mathcal{B}}\theta\left(\frac{f_n}{a}\right)d\mu_{\infty} \leq 1\right\}= 0.\end{align*}
 Therefore, $ \lim\limits_{n \to \infty} \sup ~H\int_{\mathcal{B}}\theta(af_n)d\mu_{\infty} \leq 1.$\\
 For (iii) implies (ii) is obvious as  $ \theta$ is convex and $ \theta(0) = 0 $ for all $ x \geq 0 $ and $ 0 < \epsilon \leq 1$
 \begin{align*}
 \theta(x) &= \theta( 1- \epsilon) 0 + \epsilon\left(\frac{x}{\epsilon}\right) \\&\leq (1-\epsilon) \theta(0) + \epsilon \theta\left(\frac{x}{\epsilon}\right). \end{align*}    
    Thus $ \theta(x) \leq \epsilon \theta(\frac{x}{\epsilon}), x \geq 0,~0 < \epsilon \leq 1 .$\\
    From here (ii) implies (iii) follows easily.\\   
  For (ii) implies (i):\begin{align*}
 & \lim\limits_{n \to \infty} \sup H\int_{\mathcal{B}}\theta(af_n)d\mu_{\infty} \leq 1\\
&\implies \inf\left\{ a>0 : \lim\limits_{n \to \infty}H\int_{\mathcal{B}}\theta\left(\frac{f_n}{a}\right)d\mu_{\infty} \leq 1\right\}= 0 \\
  &\implies \lim\limits_{ n \to \infty} \inf\left\{ a>0 : H\int_{\mathcal{B}}\theta\left(\frac{f_n}{a}\right)d\mu_{\infty} \leq 1 \right\} = 0 \\
&\implies \lim\limits_{n \to \infty}||f_n||_{H_\theta} = 0. \end{align*} 
    \end{proof}
    \begin{thm} 
    $( \mathcal{H}^{\theta}(\mu_{\infty}), ||.||_{H_\theta} ) $ is Banach space.
    \end{thm}
    \begin{proof}
      Let $(f_n)_{n \geq 1 } $ be a Cauchy sequence in $ \mathcal{H}^{\theta}(\mu_{\infty}). $    We find a countable partition $ ( E_k:~k \geq 1 \} $ of measurable subset of $\mathcal{B}$ such that $\mu_{\infty}(E_k) < \infty $ for all $ k \geq 1 .$\\
    Now, in restriction $ {\mu_{\infty}}_{k}(.) = \mu_{\infty}(E_k \cap .).$\\
    As $ (f_n)_{n \geq 1} $ is a Cauchy sequence in $ \mathcal{H}^{\theta}(E_k, {\mu_{\infty}}_{k}).$  Using Lemma \ref{lemma21}, $(f_n)_{n \geq 1 } $ is a Cauchy sequence in $ L_1(\mu_{\infty}).$ 
    As $ L_1(\mu_{\infty}) $ is complete, it is convergent in $ L_1(\mu_{\infty}) $ and one can extract a subsequence which converges $\m_{\infty_{k}}$-a.e. pointwise to $f$ in $E_k.$ Using the diagonal extraction procedure, we can extract a subsequence $(f_{n_{k}})_{k \geq 1}$ which converges $\m_\infty$-a.e. pointwise to $f $ on the whole space.\\
    Let $ a > 0 .$ Then there exists integer $ N_a$ and Lemma \ref{lemma22}, $ H\int_{\mathcal{B}}\theta(a(f_m -f_n))d\mu_{\infty} \leq 1 $ for all $ m,n \geq N_a $\\
    With Fatou's Lemma this gives  $$  H\int_{\mathcal{B}}\theta(a(f_m -f))d\mu_{\infty} \leq \lim\limits_{k \to \infty} \inf H \int_{\mathcal{B}}\theta(a(f_m -f_{n_{k}}))d\mu_{\infty}  \leq 1,~\forall m \geq N_a.$$
    %Using \cite{Lp},  we get  $\int_{\mathcal{B}}\theta(a(f_m -f_n))d\mu_{\infty} \leq \lim\limits_{K \to \infty} \inf \int_{\mathcal{B}}\theta(a(f_m -f_n))d\mu_{\infty} < \infty. $\\ 
    %Therefore $  \int_{\mathcal{B}}\theta(a(f_m -f_n))d\mu_{\infty}$ is Henstock integrable.\\
    So, $ f_m -f \in \mathcal{H}^{\theta}(\m_{\infty}).$ But $ f_m \in \mathcal{H}^{\theta}(\m_{\infty}),$ so $ f \in \mathcal{H}^{\theta}(\mu_{\infty}) .$  Moreover, $$\lim_{ m \to \infty}\inf H\int_{\mcB}\theta(a(f_m -f)d \m_\infty \leq 1,~\forall a>0,$$ we have $\lim\limits_{ m \to \infty}||f_m - f||_{\mcH_\theta} =0.$   So, it is a Banach space.
        \end{proof}
    \begin{corollary}\label{cor37}
     $( \mathcal{H}^{\theta}(\mu_{\infty}), ||.||_{H_\theta})$ is a Banach space if $ H(\m_\infty) = L_1(\mu_{\infty}).$
     \end{corollary}
%  From the Corollary \ref{cor37}, we can  consider  the Orlicz class as $I_{\theta}(f) = \int_{0}^{\infty} \theta(|f|)d \m_\infty \in [0, \infty] $ and we set the Luxemburg norm as $$||f||_{H_\theta}= \inf\left\{ a >0:~I_{\theta}\left(\frac{f}{a}\right) \leq 1 \right\} \in [0,\infty]$$ for every bounded measurable function  $f: \R^{+} \to \R$ with compact support.  We put $\inf(\emptyset)= \infty.$ We will find the following proposition directly from the convexity of $\theta.$
 %\begin{prop}  %\cite[Proposition 13.2.1]{BP}
 %	Let $f,g $ be bounded measurable functions and $0< \alpha < 1.$ Then
 %	\begin{enumerate}
 %		\item $I_{\theta}(\alpha f + (1-\alpha) g) \leq \alpha I_{\theta}(f) + (1-\alpha) I_{\theta}(g).$
 %		\item $I_{\theta}(\alpha f) \leq \alpha I_{\theta}(f).$
 %		\item The set $B_{\theta}= \{f \in \mcH_{\theta}(\m_{\infty}):~I_{\theta}(f) \leq 1 \}$ is convex.
 %		\item If $ f \in I_{\theta} $ and $||f||_{H_\theta} \leq 1 $ then $ f \in B_{\theta}.$
 	%\end{enumerate}
 %\end{prop}
 %\begin{proof} 	Proofs are similar with the proof of the Proposition 13.2.1 \cite{BP}. \end{proof}
 \begin{thm} 
 The space $( \mathcal{H}^{\theta}(\mu_{\infty}), ||.||_{H_\theta})$ is a symmetric space.
 \end{thm}
 \begin{proof}
     If  $ f $ is bounded measurable function, then 
   \begin{align*}
 H_{\theta}(f) & = H\int_{\mathcal{B}} \theta(|f|)d\mu_{\infty}\\    &= H\int_{\mathcal{B}}\eta_{\theta o |f|}d\mu_{\infty} \\& = H\int_{\mathcal{B}} \mu_{\infty}\{\theta(|f|) > x \} d\mu_{\infty}\\&= H\int_{\mathcal{B}} \mu_{\infty}\{|f| > \theta^{-1}(x) \}dx \\&= H\int_{\mathcal{B}} \mu_{\infty}\{|f| > y \} \theta(y)\\&=H\int_{\mathcal{B}}\eta_{|f|}d\theta.
    \end{align*} 
     Let $ f \in \mathcal{H}^{\theta}(\mu_{\infty})$ and   let $ g $ be equimeasurable to $ f,$ that is $\eta_{|f|} =\eta_{|g|}.$ Then  we have 
 \begin{align*}
    H_{\theta}(f) &= H\int_{\mathcal{B}} \eta_{|f|}d\theta\\
    & = H\int_{\mathcal{B}}\eta_{|g|}d\theta\\&= H_{\theta}(g).
    \end{align*} 
    Therefore $ g \in \mathcal{H}^{\theta}(\mu_{\infty}) $ so, $||f||_{H_{\theta}}= ||g||_{H_{\theta}}. $ Hence, this space is symmetric. \\
    It is easy to checked that  $( \mathcal{H}^{\theta}(\mu_{\infty}), ||.||_{H_\theta})$ is a normed ideal lattice.
\end{proof}
\begin{lemma}\label{lemma23}
 For any bounded measurable function $ g $ on $\mathcal{B}, ||g||_{H_{\theta}} =0 $ if and only if $ g = 0 $ a.e.
 \end{lemma}
 \begin{proof}
    As  $ ||g||_{H_{\theta}} =0 $ if and only if $ \inf\{ a>0:~H\int_{\mathcal{B}}\theta(\frac{g}{a})d\mu_{\infty} \leq 1 \} = 0.$ 
    This gives
    \begin{align*} ||g||_{H_{\theta}} = 0 & \iff %\mbox{~if~and~only~if~} 
     H\int_{\mathcal{B}}\theta\left(\frac{g}{a}\right)d\mu_{\infty} = 0  \mbox{~for~all~}  a > 0 ~( \mbox{by Lemma \ref{lemma22}(iii)}) \\
   & \iff \theta\left(\frac{g}{a}\right) = 0~a.e. \mbox{~for~all~}  a> 0 \\
   &\iff g=0~\mu_{\infty}-a.e.
   \end{align*}
    \end{proof}
     \begin{lemma}\label{lemma24}
      If $~0 < ||f||_{H_{\theta}} < \infty $ then $ H\int_{I \subseteq \mathcal{B}} \theta\left(\frac{f}{||f||_{H_\theta}}\right) \leq A,$ where $ A = \lim S\left (\theta\left(\frac{f}{||f||_{H_\theta}}\right), D_k\right),$\\ $D_k $ is a sequence of tagged partition.
      \end{lemma}
      \begin{proof}
      Let $ ||f||_{H_{\theta}} > 0. $  Then 
      $ \int_{\mathcal{B}}\theta\left(\frac{f}{||f||_{H_\theta}}\right)d\mu_{\infty} $ is Henstock integrable. By (Ch 1, Theorem 6  \cite{Cs}), for every $ k $ there exists a gauge $\delta_k$ on $ I \subseteq \mathcal{B}$ such that $ D_1, D_2 \leq \delta_k,$ then $$\left| S\left(\theta\left(\frac{f}{||f||_{H_\theta}}\right), D_1\right)- S\left(\theta\left(\frac{f}{||f||_{H_\theta}}\right), D_2\right)\right| < \frac{1}{k}.$$
    For each $ k ,$ let $ D_k < \delta_k,$ If $ k>j,$ we have $$\left| S\left(\theta\left(\frac{f}{||f||_{H_\theta}}\right), D_k\right)- S\left(\theta\left(\frac{f}{||f||_{H_\theta}}\right), D_j\right)\right| < \frac{1}{j}.$$
    Thus $\left\{S\left(\theta\left(\frac{f}{||f||_{H_\theta}}\right), D_k\right)\right\} $ is a Cauchy sequence in $ \mathbb{R} .$\\
    Let $ A = \lim\limits_{k} S\left(\theta\left(\frac{f}{||f||_{H_\theta}}\right), D_k\right) .$ Then $ \left|S\left(\theta\left(\frac{f}{||f||_{H_\theta}}\right), D_k\right)-A\right| \leq \frac{1}{k}$ for all $k .$\\
    Let $ \epsilon > 0. $ There exists natural number $'n'$ such that $ \frac{1}{n}< \frac{\epsilon}{2}.$\\
    Assume $  D < \delta_n.$     Then
     \begin{align*}
   &\left|S\left(\theta\left(\frac{f}{||f||_{H_\theta}}\right), D\right)-A\right|\\
    &\leq \left| S\left(\theta\left(\frac{f}{||f||_{H_\theta}}\right), D\right)- S\left(\theta\left(\frac{f}{||f||_{H_\theta}}\right), D_n\right)\right|+\left|S\left(\theta\left(\frac{f}{||f||_{H_\theta}}\right), D_n\right)-A\right| \\
    &< \frac{1}{n} + \frac{1}{n}< \epsilon.
    \end{align*}
    Hence $ H\int_{I\subseteq\mathcal{B}} \theta\left(\frac{f}{||f||_{H_\theta}}\right) d\mu_{\infty} \leq A .$
    \end{proof}
    \begin{thm}\label{th312}
     For all $ f \in \mathcal{H}^{\theta}(\m_{\infty}),~g \in \mathcal{H}^{\Phi}(\m_{\infty}), $ where $ \Phi $ is complementary function of $\theta,$ then $$H\int_{\mathcal{B}}|fg|d\mu_{\infty} \leq ||f||_{H_{\theta}}||g||_{H_{\Phi}} .$$
     \end{thm}
     \begin{proof}
     If $ ||f||_{H_{\theta}} =0,~||g||_{H_{\Phi}} = 0 .$
    Then by Lemma \ref{lemma23}, gives the result.
      Let $ 0 < ||f||_{H_{\theta}},~0< ||g||_{H_{\Phi}}.$ By Young's inequality, we get
    $$ st \leq \theta(s) + \Phi(t);~s, t \geq 0.$$
    Let $ s = \frac{f}{||f||_{H_{\theta}}},~t= \frac{g}{||g||_{H_{\Phi}}}.$ Then we have   
    \begin{align*}
    &H\int_{\mathcal{B}}\frac{f}{||f||_{H_{\theta}}}\frac{g}{||g||_{H_{\Phi}}}d\mu_{\infty}\\
    & \leq  H\int_{\mathcal{B}}\theta\left(\frac{f}{||f||_{H_{\theta}}}\right)d\mu_{\infty} + H\int_{\mathcal{B}}\Phi\left(\frac{g}{||g||_{H_{\Phi}}}\right)d\mu_{\infty}\\
    & \leq A + A.
    \end{align*}
    Therefore $ \int_{\mathcal{B}} fg d\mu_{\infty} \leq 2A||f||_{H_\theta}||g||_{H_\Phi}.$\\
    In particular $ A = \frac{1}{2}$ gives $ H\int_{\Omega} |fg| d\mu_{\infty} \leq ||f||_{H_\theta}||g||_{H_\Phi}.$
\end{proof}
\section{Denseness  of $C_{0}^{\infty}(\mathbb{R}^n)$}
\begin{lemma}\label{lemma25}
  Let  $\M $ be an element of $ C_{0}^{\infty} (\mathbb{R}^n)$  satisfying $ \M \geq 0 $ and $ \int \M(t)d \m_\infty(t) =1 .$ Define a  sequence $ (\M_k)_{k\ge 1}$ of elements of $ C_{0}^{\infty} (\mathbb{R}^n)$ by $ \M_{k}(t) =k\M(kt).$ Let $ \theta $ be an $N$-function and $ f \in \mathcal{H}^{\theta}(\mu_{\infty}), $ then the convolution $ \M_{k} \ast f \in \mathcal{H}^{\theta}(\mu_{\infty}) $ and $||\M_{k}\ast f - f||_{H_{\theta}} \to 0 $ as $ k \to \infty .$
  \end{lemma}
  \begin{proof}
   Let $ \Phi $ be the complementary $N$-function of $ \theta$ and let $ g \in \mathcal{H}^{\Phi}(\mu_{\infty}) $ with $ ||g||_{H_\Phi} = 1.$\\
    Then
     \begin{align*}
     &\int|\M_{k}\ast f(x) -f(x)||g(x)|d\m_\infty(x) \\
     &= \int|\M_k\ast f(x)- \int\M(t)d\mu_{\infty}(t) f(x)||g(x)|d\m_\infty(x)\\
     &=\int|\int\M_k(t)f(x-t)d\mu_{\infty}(t) - \int\M(t)d\mu_{\infty}(t) f(x)||g(x)|d\m_\infty(x)\\
     &\le\int|\int\M_k(t)f(x-t)d\mu_{\infty}(t) - \int\M_k(t)d\mu_{\infty}(t) f(x)||g(x)|d\m_\infty(x)\\
     & \leq \int\left\{\int|f(x-t) -f(x)||g(x)|d \m_\infty(x)\right\}\M_{k}(t)d\m_\infty(t) \\& \leq 2 \int ||f_t-f||_{H_\theta} \M_k(t)d \m_\infty(t) ~(\mbox{by Theorem \ref{th312}}), 
     \end{align*} 
 where $ f_t(x) = f(x-t).$\\
 By assumption $ ||g||_{H_\Phi} = 1,$ therefore by Theorem \ref{th312}, we have 
 \begin{align*}
 ||\M_{k} \ast f-f||_{H_\theta} & \leq 2 \int||f_{t} - f||_{H_\theta} \M_k(t)d \m_\infty (t)  \\ & =  2 \int||f_{\frac{t}{k}}-f||_{H_\theta} \M(t)d \m_\infty(t).
\end{align*} 
Since, $ f \in \mathcal{H}^{\theta}(\mu_{\infty})$ and $\M $ has compact support so, for every $ \epsilon > 0 $ there exists $ k$ sufficiently large such that
\begin{align*}
\int||f_{\frac{t}{k}}-f||_{H_\theta} \M(t)d\m_\infty(t) & \leq \epsilon\int \M(t)d\m_\infty(t)\\
& = \epsilon.
\end{align*}  
 Therefore $||\M_{k} \ast f-f||_{H_\theta} \to 0 $ as $ k \to \infty.$ It is clear that  $\mathbb{M}_k\ast f\in \mathcal{H}^{\theta}(\mu_{\infty}).$ % because $ f \in \mathcal{H}^{\theta}(\mu_{\infty}).$ 
 \end{proof}
 \begin{thm}
 If $\mathcal{B} = \mathbb{R}^n,$ then $C_{0}^{\infty}(\mathbb{R}^n) $ is dense in $  \mathcal{H}^{\theta}( \mathbb{R}^n)=\mathcal{H}^{\theta}( \mathbb{R}^n, \mu_{\infty}).$
 \end{thm}
 \begin{proof}
  Suppose  $ f \in \mathcal{H}^{\theta}(\mu_{\infty}) $ has  compact support. Let us assume a sequence $( \M_n) $ of elements of $ C_{0}^{\infty} $ such that $ \M_{n}(x) = 1 $ for $ x \leq n $ and $ \M_{n}(x) = 0 $ for $ x \geq 2n.$ \\
   Clearly $ \M_n \ast f $ is a function with compact support, and $\M_n \ast f  \to f \in \mathcal{H}^{\theta}(\mathbb{R}^n) .$\\
   Let $ (\M_k) $ be the sequence defined in Lemma \ref{lemma25}. Define    \begin{align*}
   f_k(x) & = f \ast \M_k(x) \\
   & = \int f(t)\M_k(x-t)d\m_\infty(t).
   \end{align*}
   Derivative of distributional  function is again distributional function, (page 33 \cite{MA} and Lemma 5.7 of \cite{JJ}) shows that $ D^{\alpha} f\in \mathcal{H}^{\theta}(\mathbb{R}^n). $    So, for all $ \alpha , |\alpha| \leq m $ and $ \M_k \in C_{0}^{\infty}.$\\
    We have  \begin{align*}
    D^{\alpha}f_k(x) &= D^{\alpha}\int f(t)\M_k (x -t) d\m_\infty(t)\\
    & = \int f(t)D^{\alpha} \M_k(x-t)d\m_\infty(t)\\
    & = \int D^{\alpha} f(t) \M_k(x-t) d\m_\infty(t).
    \end{align*}
     So, $||D^{\alpha}f_k - D^{\alpha} f ||_{H_\theta} = || D^{\alpha} f \ast \M_k- D^{\alpha}f ||_{H_\theta} \to 0 $ as $ k \to \infty.$      Therefore, $ C_{0}^{\infty}(\mathbb{R}^n) $ is dense in $ \mathcal{H}^{\theta}(\mathbb{R}^n, \mu_{\infty}).$
     \end{proof}

     \section*{Acknowledgement} The authors would like to thank the reviewer for reading the manuscript carefully and making      valuable suggestions that significantly improve the presentation of the paper. %We highly indebted to The Reviewer's  for his guidance and  suggestion as well as providing necessary information to improved  the article.

\end{document}